\documentclass{gen-j-l}
\usepackage{amscd,amssymb}  
\usepackage{latexsym,stmaryrd,mathrsfs,enumerate,scalefnt}
\usepackage{tikz}
\usepackage{color}
\usetikzlibrary{calc}
\usepackage[colorlinks=true,urlcolor=black,linkcolor=black,citecolor=black]{hyperref}
\usepackage{acronym}
\newcommand{\lb}{\ensuremath{\llbracket}}
\newcommand{\rb}{\ensuremath{\rrbracket}}
\newcommand{\<}{\ensuremath{\langle}}
\newcommand{\cB}{\ensuremath{\mathcal{B}}}
\renewcommand{\>}{\ensuremath{\rangle}}

\theoremstyle{plain}
\newtheorem{theorem}{Theorem}[section]
\newtheorem{corollary}[theorem]{Corollary}
\newtheorem{lemma}[theorem]{Lemma}

\theoremstyle{definition}

\newcounter{claim}

\newcounter{conjecture}

\newcounter{prob}

\theoremstyle{remark}
\newtheorem*{remark}{Remark}

\numberwithin{theorem}{section}
\numberwithin{claim}{section}
\numberwithin{equation}{section}
\numberwithin{conjecture}{section}

\newcommand{\Peter}{P{\'e}ter}
\newcommand{\Palfy}{P\'alfy}
\newcommand{\Pudlak}{Pudl\'ak}

\newcommand{\defn}[1]{\emph{#1}}

\renewcommand{\leq}{\ensuremath{\leqslant}}
\renewcommand{\nleq}{\ensuremath{\nleqslant}}
\renewcommand{\geq}{\ensuremath{\geqslant}}

\newcommand{\subnormal}{\ensuremath{\trianglelefteqslant}}

\newcommand{\meet}{\ensuremath{\wedge}}
\newcommand{\join}{\ensuremath{\vee}}
\newcommand{\Meet}{\ensuremath{\bigwedge}}

\newcommand{\Aut}{\ensuremath{\operatorname{Aut}}}
\newcommand{\Eq}{\ensuremath{\operatorname{Eq}}}

\newcommand{\Sub}{\ensuremath{\operatorname{Sub}}}
\newcommand{\Sym}{\ensuremath{\operatorname{Sym}}}
\newcommand{\core}{\ensuremath{\operatorname{core}}}

\newcommand{\sG}{\ensuremath{\mathfrak{X}}}

\newcommand{\G}{\ensuremath{\mathfrak{G}}}
\newcommand{\bH}{\ensuremath{\mathbf{H}}}

\newcommand{\sK}{\ensuremath{\mathscr{K}}}

\newcommand{\sL}{\ensuremath{\mathscr{L}}}
\newcommand{\fL}{\ensuremath{\mathfrak{L}}}

\newcommand{\sP}{\ensuremath{\mathscr{P}}}
\newcommand{\cP}{\ensuremath{\mathcal{P}}}

\newcommand{\IE}{{\small IE}}

\renewcommand{\phi}{\ensuremath{\varphi}}

\begin{document}

\title{Interval enforceable properties of finite groups}

\author{William DeMeo}

\address{Department of Mathematics, University of South Carolina, Columbia, SC 29208, United States}

\email{williamdemeo@gmail.com}
\urladdr{http://williamdemeo.org}



\thanks{The author wishes to thank the anonymous referee for many
excellent suggestions that greatly improved the presentation of the paper.}

\begin{abstract}

We propose a classification of group properties according to whether they can be deduced from the assumption that a group's subgroup lattice contains an interval isomorphic to some lattice. We are able to classify a few group properties as being ``interval enforceable'' in this sense, and we establish that other properties satisfy a weaker notion of ``core-free interval enforceable.''  We also show that if there exists a group property and its negation that are both core-free interval enforceable, this would settle an important open question in universal algebra.

\end{abstract}

\maketitle

\newacro{FLRP}{{\it Finite Lattice Representation Problem}}
\newacro{IE}{{\it interval enforceable}}
\newacro{cfIE}[cf-IE]{{\it core-free interval enforceable}}
\newacro{minIE}[min-IE]{{\it minimal interval enforceable}}
\newacro{CFSG}{Classification of Finite Simple Groups}


\section{Introduction}
\label{sec:intro}
The study of subgroup lattices has a long history that began with
Richard Dedekind~\cite{Dedekind:1877} 
and 
Ada Rottlaender~\cite{Rottlaender:1928}, and
continued with important contributions by Reinhold Baer, 
{\O}ystein Ore, 
Michio Suzuki, 
Roland Schmidt, 
and many others (see Schmidt~\cite{Schmidt:1994}).
Much of this work focuses on the problem of deducing
properties of a group $G$ from assumptions about the structure of its lattice of
subgroups, $\Sub(G)$,  or, conversely, deducing lattice theoretical properties
of $\Sub(G)$ from assumptions about $G$. 

Historically, less attention was paid to the local structure of the
subgroup lattice of a finite group, perhaps because it seemed that very little
about $G$ could be inferred from knowledge of, say, an \emph{upper
  interval}, $\lb H,G \rb = \{K \mid H\leq K \leq G\}$,
in the subgroup lattice of $G$.
Recently, however, this topic has attracted more attention (see, e.g.,
\cite{Aschbacher:2009,Lucchini:1997,Basile:2001,Borner:1999,Kohler:1983,Lucchini:1994a,Palfy:1988,Palfy:1995}),
mostly owing to its connection with one of the most important open
problems in universal algebra,
the \ac{FLRP}. This is the problem of
characterizing the lattices that are (isomorphic to) congruence lattices of
finite algebras (see, e.g., \cite{Berman:1970,DeMeo:thesis,Palfy:1995,Palfy:2001}). 
There is a remarkable theorem relating this problem to intervals in subgroup
lattices of finite groups. 
\begin{theorem}[\Palfy\ and \Pudlak~\cite{Palfy:1980}]
\label{thm:P5}
The following statements are equivalent:
\begin{enumerate}[(A)]
\item Every finite lattice is isomorphic to
  the congruence lattice of a finite algebra.
\item Every finite lattice is isomorphic to
  an interval in the subgroup lattice of a finite group.
\end{enumerate}
\end{theorem}
If these statements are true (resp., false), then we say the \acs{FLRP} has
a positive (resp., negative) answer. 
Thus, if we can find a finite lattice $L$ for which it can be proved that there
is no finite group $G$ with $L \cong \lb H,G \rb$ for some $H< G$, then the
\acs{FLRP} has a negative answer.  

In this paper we propose a new classification of group properties according to
whether or not they can be deduced from the assumption that $\Sub(G)$ has an upper
interval isomorphic to some finite lattice.  We believe that discovering which 
group properties can (or cannot) be connected to the local structure
of a subgroup lattice is itself a worthwhile endeavor, but we will also describe 
how this classification could provide a solution of the \acs{FLRP}.  

Suppose $\cP$ is a \emph{group theoretical property}\footnote{This and other
  italicized terms in the introduction will be defined more formally in
  Section~\ref{sec:notation-definitions}.}  
and suppose there exists a finite lattice $L$ such that if $G$ is a finite group
with $L \cong \lb H,G \rb$ for some $H\leq G$, then $G$ has property $\cP$.  We call
such a property $\cP$ \ac{IE}.  If the lattice involved is germaine to the
discussion, we say that $\cP$ is  \emph{interval enforceable by} $L$.
An \defn{interval enforceable class of groups} is a class of groups all of which
have a common interval enforceable property.

Although it depends on the lattice $L$, generally speaking it is difficult to 
deduce very much about a group $G$ from the assumption that an upper interval
in $\Sub(G)$ is isomorphic to $L$.  It becomes easier easier if, in addition to the
hypothesis $L\cong\lb H,G \rb$, we assume that the subgroup $H$ is
\emph{core-free} in $G$; that is, $H$ contains no nontrivial normal subgroup of
$G$.  Properties of $G$ that can be deduced from these assumptions are what we
call \ac{cfIE}. 

Extending this idea, we consider finite collections $\sL$ of finite lattices
and ask what can be proved about a group $G$ if one assumes that each 
$L_i\in \sL$ is isomorphic to an upper interval $\lb H_i, G \rb\leq \Sub(G)$, with
each $H_i$ core-free in $G$.  Clearly, if $\Sub(G)$ has such upper intervals,
and if corresponding to each $L_i\in \sL$ there is a property
$\cP_i$ that is \ac{cfIE} by $L_i$, then $G$ must have all of the properties
$\cP_i$. A related question is the following: Given a set $\sP$ of
\ac{cfIE} properties, is the conjunction $\Meet \sP$ \ac{cfIE}?  
Corollary \ref{cor:isle-prop-groups-1} answers this question affirmatively. 

In this paper, we will identify some group properties that are \ac{cfIE}, and
others that are not. We will see that the \ac{cfIE} properties
found thus far are negations of common group properties (for
example, ``not solvable,'' ``not almost simple,'' ``not alternating,'' ``not
symmetric'').  Moreover, we prove that in these special cases the
corresponding group properties (``solvable,'' ``almost simple,''
``alternating,''  ``symmetric'') that are not \ac{cfIE}.  This and
other considerations suggest that a group property and its
negation cannot both be \ac{cfIE}.  As yet, we are unable to prove this.
A related question is whether, for every group property
$\cP$, either $\cP$ is \ac{cfIE} or $\neg \cP$ is \ac{cfIE}. 

Our main result (Theorem~\ref{thm-wjd-1}) connects the
foregoing ideas with the \acs{FLRP}, as follows:\\[6pt]
{\it 
Statement (B) of Theorem~\ref{thm:P5} is equivalent to the following statement:

\begin{enumerate}
\item[(C)]
Fix $n\geq 2$ and let $\sL = \{L_1, \dots, L_n\}$ be any collection of
finite lattices at least two of which have more than two elements.
For each $i = 1, 2, \dots, n$, let $\sG_i$ denote the class that is
core-free interval enforcable by $L_i$. Then there exists a finite group $G \in
\bigcap\limits_{i=1}^n \sG_i$ such that for 
each $L_i \in \sL$ we have $L_i\cong \lb H_i, G \rb$ for some subgroup
$H_i$ that is core-free in $G$. 
\end{enumerate}

\begin{remark}
By (C), the \acs{FLRP} would have a negative answer if we
could find a collection $\sG_1, \dots, \sG_n$ of \acs{cfIE} classes
such that $\bigcap\limits_{i=1}^n \sG_i$ is empty.
\end{remark}
}

Core-free interval enforceable properties are related to
permutation representations of groups.
If $H$ is a core-free subgroup of $G$, then $G$ has a faithful permutation 
representation $\phi:G\hookrightarrow \Sym(G/H)$.
Let $\<G/H, \phi(G)\>$ denote the algebra comprised of the right cosets
$G/H$ acted upon by right multiplication by elements of $G$; that is,
$\phi(g): Hx \mapsto Hxg$.  It is
well known that the congruence lattice of this algebra (i.e., the lattice of
systems of imprimitivity) 
 is isomorphic to the interval $\lb H, G \rb$ in the subgroup
lattice of $G$.\footnote{See \cite[Lemma 4.20]{alvi:1987}
or~\cite[Theorem 1.5A]{Dixon:1996}.}
This puts statement (C) into perspective.
If the \acs{FLRP} has a positive answer, then no matter 
what we take as our finite collection $\sL$---for example, we
might take $\sL$ to be \emph{all} finite lattices with
at most $N$ elements for some large $N< \omega$---we can always find a \emph{single}
finite group $G$ such that every lattice in $\sL$ is isomorphic to the interval
in $\Sub(G)$ above a core-free subgroup.
As a result, this group $G$ must have so many faithful
representations  $G\hookrightarrow \Sym(G/H_i)$ with systems of imprimitivity
isomorphic to $L_i$,
one such
representation for each distinct $L_i\in \sL$.  Moreover, the group $G$ having
this property can be chosen from the class $\bigcap\limits_{i=1}^n \sG_i$, where 
$\sG_1, \dots, \sG_n$ is an arbitrary collection of \acs{cfIE} classes of groups.

\section{Notation and definitions}
\label{sec:notation-definitions}
In this paper, \emph{all groups and lattices are finite}.  We use 
$\G$ to denote the class of all finite groups.
Given a group $G$, we denote the set of subgroups of $G$ by $\Sub(G)$.  The
algebra $\<\Sub(G), \meet, \join\>$ is a lattice where the $\meet$ (``meet'') and
$\join$ (``join'') operations are defined for all $H$ and $K$ in $\Sub(G)$ by
$H\meet K = H\cap K$ and $H\join K = \<H, K\> = $ the smallest subgroup of $G$
containing both $H$ and $K$.  We will refer to the set
$\Sub(G)$ as a lattice, without explicitly mentioning the $\meet$ and
$\join$ operations.

By $H \leq G$ (resp.,
$H < G$) we mean $H$ is a subgroup (resp., proper subgroup) of $G$.
For $H\leq G$, the
\emph{core of $H$ in $G$}, denoted by $\core_G(H)$, is the largest normal subgroup of $G$
contained in $H$.
If $\core_G(H)=1$, then we say that $H$ is \emph{core-free in $G$}.
For $H\leq G$,
by the \defn{interval} $\lb H, G \rb$ we mean 
the set $\{K \mid H\leq K \leq G\}$, which is a
sublattice of $\Sub(G)$.
With this notation, $\Sub(G)=\lb 1,G \rb$.
When viewing $\lb H,G \rb$ as a
sublattice of $\Sub(G)$, we sometimes refer to it as an \defn{upper interval}. 
Given a lattice $L$ and a group $G$, the expression $L \cong \lb H, G \rb$ will
mean that there exists a subgroup $H \leq G$ such that $L$ is isomorphic to the
interval $\{K \mid H\leq K \leq G\}$ in the subgroup lattice of $G$.

By a \defn{group theoretical class}, or \defn{class of groups}, we mean a
collection $\sG$ of groups that is closed under isomorphism:
if $G_0\in \sG$ and  $G_1\cong G_0$, then $G_1\in \sG$.
A \defn{group theoretical property}, or simply \defn{property of groups},
is a property $\cP$ such that if a group $G_0$ has property $\cP$ and
$G_1\cong G_0$, then $G_1$ has property $\cP$.\footnote{It seems there
  is no single standard definition of \emph{group theoretical class}.
  While some authors (e.g.,~\cite{Doerk:1992}, \cite{BBE:2006}) use the same
  definition we use here, others (e.g.~\cite{Robinson:1996}, \cite{Rose:1978})
  require that every group theoretical class contains the one element group. 
  In the sequel we consider negations of group properties, and we would
  like these to qualify as group properties.  Therefore, we don't require
  that every group theoretical class contains the one element group.}   
Thus if $\sG_{\cP}$ denotes the collection of all groups having the group
property $\cP$, then  $\sG_{\cP}$  is a class of groups, and belonging to a
particular class of groups is a group theoretical property.

If $\sK$ is a class of algebras (e.g., a class of groups), then we say that
$\sK$ is \emph{closed under homomorphic images} and we write $\bH(\sK) = \sK$
provided $\phi(G)\in \sK$ whenever $G\in \sK$ and $\phi$ is a homomorphism of
$G$.

Let $\fL$
denote the class of all finite lattices, and $\G$ the class of all 
finite groups. Let $\cP$ be a group theoretical property and $\sG_\cP$
the associated class of all groups with property $\cP$.  
We call $\cP$ (and $\sG_\cP$)
\begin{itemize}
\item 
\acf{IE} provided
\[
(\exists L\in \fL)  \; (\forall G \in \G) \; \bigl(L\cong \lb H,G \rb \; \longrightarrow \; G
\in \sG_\cP\bigr)
\]
\item
\acf{cfIE} provided
\[
(\exists L\in \fL)  \; (\forall G\in \G) \; \bigl( L\cong \lb H,G \rb \; \Meet \; \core_G(H)=1
\; \longrightarrow \; G  \in \sG_\cP \bigr)
\]
\item 
\acf{minIE}
provided there exists $L\in \fL$ such that if $L\cong \lb H,G \rb$ for some
group $G\in \G$ of minimal order (with respect to $L\cong \lb H,G \rb$), then
$G \in \sG_\cP$.
\end{itemize}
In this paper we will have little to
say about min-\IE\ properties.  Nonetheless, we include this class in our list
of new definitions because properties of this type arise often (see, e.g.,
\cite{Lucchini:1994a}), and a primary aim of this paper is to formalize
various notions of interval enforceability that we believe are
useful in applications. 

\section{Results}
Clearly, if $\cP$ is an interval enforceable property, then it is also
core-free interval enforceable.  There is an easy
sufficient condition under which the converse holds.  
Suppose $\cP$ is a group property, let $\sG_{\cP}$  denote the
class of all groups with property $\cP$, and let
 $\sG_{\cP}^c$ denote the class of all groups that do not have property $\cP$.
\begin{lemma}
\label{lemma-wjd-2}
Suppose $\cP$ is a core-free interval enforceable property.  
If $\bH(\sG_{\cP}^c) = \sG_{\cP}^c$, then $\cP$ is an interval enforceable property.
\end{lemma}
\begin{proof}
Since $\cP$ is \acs{cfIE}, there is a lattice $L$ such that
\begin{equation}
  \label{eq:100}
L \cong \lb H,G \rb \; \Meet \; \core_G(H)=1 \; \longrightarrow \; G\in \sG_\cP.
\end{equation}
Under the assumption $\bH(\sG_\cP^c) = \sG_\cP^c$ we prove
\begin{equation}
  \label{eq:200}
L \cong \lb H,G \rb \; \longrightarrow \; G\in \sG_\cP.
\end{equation}
If~(\ref{eq:200}) fails, then there is a
group $G\in \sG_{\cP}^c$ with $L\cong \lb H,G \rb$.  Let $N = \core_G(H)$.  Then $L \cong
\lb H/N,G/N \rb$ and $H/N$ is core-free in $G/N$ so, by hypothesis~(\ref{eq:100}),
$G/N \in \sG_\cP$.  But $G/N \in \sG_{\cP}^c$, since $\sG_{\cP}^c$ is closed under homomorphic images.
\end{proof}

In \cite{Palfy:1995}, \Peter\ \Palfy\ gives an example of a lattice that cannot occur as an
upper interval in the subgroup lattice finite solvable group.  (We give other examples
in Section~\ref{sec:parachute-lattices}.) 
In his Ph.D.~thesis~\cite{Basile:2001}, Alberto Basile proves that if
$G$ is an alternating or symmetric group, then there are certain lattices that
cannot occur as upper intervals in $\Sub(G)$. Another class of lattices with
this property is described by Aschbacher and Shareshian in~\cite{Aschbacher:2009}. 
Thus, two classes of groups that are known to be at least \acs{cfIE} are the following:
\begin{itemize}
\item $\sG_0 = \mathfrak{S}^c = $ nonsolvable finite groups;
\item $\sG_1 =\bigl\{G\in \G \mid (\forall n<\omega) \; \bigl(G \neq A_n \Meet  G\neq S_n\bigr) \bigr\}$,
\end{itemize}
where $A_n$ and $S_n$ denote, respectively,
the alternating and symmetric groups on
$n$ letters.
Note that both classes $\sG_0$ and $\sG_1$ satisfy the hypothesis of \ref{lemma-wjd-2}.
Explicitly, $\sG_0^c = \mathfrak{S}$, the class of solvable groups, is closed under homomorphic
images, as is the class $\sG_1^c$ of alternating and symmetric groups. 
Therefore, by Lemma~\ref{lemma-wjd-2}, $\sG_0$ and $\sG_1$ are \IE\ classes.
By contrast, suppose 
there exists a finite lattice $L$ such that
 \[
L \cong \lb H, G \rb \; \Meet \; \core_G(H)=1 \; \longrightarrow \; G
 \text{ is subdirectly irreducible.}  
\]
Lemma~\ref{lemma-wjd-2} does not apply in this case since the class of
subdirectly reducible groups is obviously not closed under homomorphic 
images.\footnote{Recall, for groups \defn{subdirectly irreducible} is equivalent
  to having a unique minimal normal subgroup.
 Every algebra, in particular every group $G$, has a subdirect
  decomposition into subdirectly irreducibles, say, $G\hookrightarrow G/N_1 \times \cdots\times
  G/N_n$, so there are always  subdirectly irreducible homomorphic images.}
In Section \ref{sec:parachute-lattices} 
below we describe
lattices with which we can prove that the following classes are at least 
\acs{cfIE}: 
\begin{itemize}
\item $\sG_2 = $ the subdirectly irreducible groups;
\item $\sG_3 = $ the groups having no nontrivial abelian normal subgroups;
\item $\sG_4 = \{G\in \G \mid C_G(M) = 1 \text{ for all } 1\neq M\subnormal G\}$.
\end{itemize}

We noted above that $\sG_2$ fails to satisfy the hypothesis of
\ref{lemma-wjd-2}. The same can be said of $\sG_3$ and $\sG_4$. That is, 
$\bH(\sG_i^c) \neq \sG_i^c$ for $i= 2, 3, 4$.  To verify this take $H\in
\sG_i$, $K\in \sG_i^c$, and consider $H\times K$.  In each case ($i=2, 3, 4$) we
see that $H\times K$ belongs to $\sG_i^c$, but the homomorphic image
$(H\times K)/(1\times K) \cong H$ does not. 

\subsection{Negations of interval enforceable properties}
\label{sec:negat-interv-enforc}
If a lattice $L$
is isomorphic to an interval in the subgroup lattice of a finite group, then we call
$L$ \defn{group representable}.  Recall, Theorem~\ref{thm:P5} says that the
\acs{FLRP} has a negative answer if we can find a finite lattice that is not group
representable. 

Suppose there exists a property $\cP$ such that both $\cP$ and
its negation $\neg \cP$ are interval enforceable by the lattices $L$ and $L_c$,
respectively.  That is $L\cong \lb H,G \rb$ implies $G$ has property $\cP$, and 
$L_c\cong \lb H_c,G_c \rb$ implies $G_c$ does not have property $\cP$.  
Then clearly the lattice in Figure~\ref{fig:twopanelchute} could not be group
representable.   
\begin{figure}[!h]
  \centering
\begin{tikzpicture}[scale=0.7]
  \node (G) at (0,6.25) [fill,circle,inner sep=1.2pt] {};
  \node (K1) at (-1.75,4) [fill,circle,inner sep=1.2pt] {};
  \node (K2) at (1.75,4) [fill,circle,inner sep=1.2pt] {};
  \node (H) at (0,2) [fill,circle,inner sep=1.2pt] {};

\draw (-.93,5.2) node {$L$};
\draw (.93,5.2) node {$L_c$};

\draw[semithick] 
   (K1) to (H) to (K2)
   (G) to [out=197,in=85] (K1) 
   (K1) to [out=15,in=-95] (G)
   (G) to [out=-15,in=95] (K2) 
   (K2) to [out=165,in=-85] (G);

\end{tikzpicture}
\caption{}
\label{fig:twopanelchute}  
\end{figure}
As the next result shows, however, if a group property and its
negation are interval enforceable by the lattices $L$ and $L_c$, then already
at least one of these lattices is not group representable.
\begin{lemma}
\label{lemma:ie-prop-and-neg}
  If $\cP$ is a group property that is interval enforceable by a group
  representable lattice, then it is not the case that $\neg \cP$ is interval
  enforceable by a group representable lattice. 
\end{lemma}
\begin{proof}
Assume $\cP$ is 
interval enforceable by the group representable lattice $L$, and let
$H\leq G$ be groups for which $L\cong \lb H, G\rb$.
If $\neg \cP$ is interval enforceable by the group representable
lattice $L_c$, then there exist $H_c\leq G_c$ satisfying
$L_c\cong \lb H_c, G_c\rb$. Consider the group $G\times G_c$. This
has upper intervals $L\cong \lb H\times G_c, G\times G_c \rb$ and 
$L_c\cong \lb G\times H_c, G\times G_c \rb$ and therefore, by the interval
enforceability assumptions,  the group $G\times G_c$ has the properties
$\cP$ and $\neg \cP$ simultaneously, which is a contradiction.
\end{proof}
To take a concrete example, nonsolvability is \IE.  However, solvability is
obviously not \IE. For, if $L\cong \lb H, G \rb$ then for any nonsolvable 
group $K$ we have $L\cong \lb H\times K, G\times K \rb$, and of course $G\times K$ is
nonsolvable.  Note that here (and in the proof 
of Lemma~\ref{lemma:ie-prop-and-neg}) the group $H\times K$ at the bottom of
the interval is not core-free.  So a more interesting question is whether a
property and its negation can both be \acs{cfIE}.  Again, if such a property were
found, a lattice of the form in Figure~\ref{fig:twopanelchute} would give a
negative answer to the \acs{FLRP}, though this requires additional justification to address
the core-free aspect (see Section \ref{sec:parachute-lattices}).  

This leads to the following question:
If $\cP$ is core-free interval enforceable by a group representable lattice,
does it follow that $\neg \cP$ is not core-free interval enforceable by a group
representable lattice?
We provide an affirmative answer in some special cases, such as when $\cP$ means
``not solvable'' or ``not almost simple.''
Indeed, Lemma~\ref{lem:IE-must-have-wreaths} implies that the class of
solvable groups, and more generally any class of groups that omits certain wreath
products, cannot be core-free interval enforceable by a group representable
lattice. 

\begin{lemma}
\label{lem:IE-must-have-wreaths}
Suppose $\cP$ is core-free interval enforceable by a group
representable lattice. Then, for any finite nonabelian simple group $S$, there
exists a wreath product group of the form $W = S\wr \bar{U}$ that has property
$\cP$. 
\end{lemma}

\begin{proof}
  Let $L$ be a group representable lattice such that if $L\cong \lb H,G \rb$ and
  $\core_G(H)=1$ then $G\in \sG_\cP$.
  Since $L$ is group representable, there exists a $\cP$-group $G$ with $L
  \cong \lb H,G \rb$. 
  We apply an idea of Hans Kurzweil (see~\cite{Kurzweil:1985}) twice.
Fix a finite nonabelian simple
  group $S$. Suppose the index of $H$ in $G$ is $|G:H| = n$.
  Then the action of $G$ on the cosets of $H$ induces an automorphism of the
  group $S^n$ by permutation of coordinates.  Denote this representation by
  $\phi: G \rightarrow \Aut(S^n)$, 
  and let the image of $G$ be $\phi(G) =
  \bar{G} \leq \Aut(S^n)$.  
  The wreath product under this action is the group
  \[
  U:= S\wr_\phi G = S^n \rtimes_\phi G = S^n \rtimes \bar{G}, 
  \]
  with multiplication given by
  \[
  (s_1, \dots, s_n, x) (t_1, \dots, t_n, y) = 
  (s_1 t_{x(1)}, \dots, s_nt_{x(n)}, x y),
  \]
  for $s_i, t_i \in S$ and $x, y \in \bar{G}$.  (For the remainder of the proof,
  we suppress the semidirect product symbol and write, for example, $S^n\bar{G}$
  instead of $S^n \rtimes \bar{G}$.)

  An illustration of the subgroup lattice of such a wreath product appears in
  Figure~\ref{fig:kurzweil}.  Note that the interval
  $\lb D, S^n \rb$, where $D$ denotes the diagonal subgroup of
  $S^n$, is isomorphic to $\Eq(n)'$, the dual of the lattice of partitions of an
  $n$-element set.
  The dual lattice $L'$ is an upper interval of $\Sub(U)$, namely,
  $L'\cong \lb D\bar{G}, U \rb$.\footnote{These facts, which were proved by Kurzweil
    in~\cite{Kurzweil:1985}, are discussed in greater detail in~\cite[Section 2.2]{DeMeo:thesis}.} 

  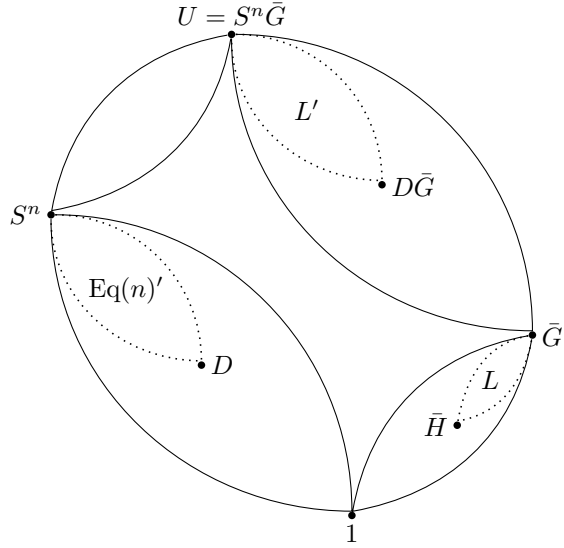
\begin{figure}[!h]
\begin{center}
  \begin{tikzpicture}[scale=.8]
    \node (G) at (3,3) [fill,circle,inner sep=1pt] {}; 
    \draw (G) node [right] {$\bar{G}$};
    \node (H) at (1.75,1.5) [fill,circle,inner sep=1pt] {}; 
    \draw (H) node [left] {$\bar{H}$};
    \node (Sn) at (-5,5) [fill,circle,inner sep=1pt] {}; 
    \draw (Sn) node [left] {$S^n$};
    \node (D) at (-2.5,2.5) [fill,circle,inner sep=1pt] {}; 
    \draw (D) node [right] {$D$};
    \node (DG) at (0.5,5.5) [fill,circle,inner sep=1pt] {}; 
    \draw (DG) node [right] {$D \bar{G}$};
    \node (1) at (0,0) [fill,circle,inner sep=1pt] {}; 
    \draw (1) node [below] {$1$};
    \node (SnG) at (-2,8) [fill,circle,inner sep=1pt] {}; 
    \draw (SnG) node [above] {$U=S^n \bar{G}$};
   \draw
    (G) to [out=190,in=80] (1) to [out=10,in=-100] (G)
    (Sn) to [out=0,in=90] (1) to [out=180,in=-90] (Sn)
    (SnG) to [out=190,in=80] (Sn) to [out=10,in=-100] (SnG)
    (SnG) to [out=0,in=90] (G) to [out=180,in=-90] (SnG);
    \draw[dotted, semithick]
    (G) to [out=190,in=80] (H) to [out=10,in=-100] (G)
    (SnG) to [out=0,in=90] (DG) to [out=180,in=-90] (SnG)
    (Sn) to [out=0,in=90] (D) to [out=180,in=-90] (Sn);
    \draw 
    (-3.75,3.75) node {$\Eq(n)'$}
    (-.75,6.75) node {$L'$}
    (2.3,2.25) node {$L$};
  \end{tikzpicture}
\end{center}
    \caption{Hasse diagram illustrating some features of the subgroup lattice of
      the wreath product $U$.}
    \label{fig:kurzweil}
  \end{figure}
  It is important to note (and we prove below) that if $H$ is core-free in $G$ --
  equivalently, if $\ker \phi = 1$ -- then the foregoing construction results in
  the subgroup $D\bar{G}$ being core-free in $U$.  Therefore, by repeating the
  foregoing procedure, with  
  $H_1 = D\bar{G}$ denoting the (core-free) subgroup of $U$ such that $L' \cong
  \lb H_1, U \rb$, we find that $L = L''\cong \lb D_1 \bar{U}, S^m\bar{U} \rb$, where $m = |U:H_1|$, and $D_1$ denotes the diagonal subgroup of $S^m$.
    Since $D_1\bar{U}$ will be core-free in $S^m \bar{U}$ 
    then, it follows by the original hypothesis that $S^m \bar{U} = S \wr
    \bar{U}$ must have property $\cP$.

  To complete the proof, we check that starting with a core-free subgroup
  $H \leq G$ in the Kurzweil construction just described results in a
  core-free subgroup $D \bar{G} \leq U$.   Let $N = \core_U(D\bar{G})$.  Then, for all $w=(d,\dots, d, x) \in N$ and for all 
  $u = (t_1,\dots, t_n, g)\in U$, we have $u w u^{-1}\in N$. 
  Fix $w=(d,\dots, d, x) \in N$.  We will choose $u\in U$ so that
  the condition $u w u^{-1}\in N$ implies $x$ acts trivially on $\{1, \dots, n\}$.
  First note that if $u = (t_1,\dots, t_n, 1)$, then
  \begin{align*}
  u w u^{-1} &= (t_1,\dots, t_n, 1) (d, \dots, d, x) (t_1^{-1},\dots, t_n^{-1}, 1)\\
  &=(t_1 d \,t_{x(1)}^{-1},\dots, t_nd \,t_{x(n)}^{-1}, 1) \in N,
  \end{align*}
  and this implies that $t_1 d\, t_{x(1)}^{-1} = t_2 d\, t_{x(2)}^{-1} =\cdots = t_nd \,t_{x(n)}^{-1}$. 
  Suppose by way of contradiction that $x(1) = j\neq 1$.  Then, since $x$ is a
  permutation (hence, one-to-one), $x(k) \neq j$ for
  each $k\in \{2, 3, \dots, n\}$.  Pick one such $k$ other than $j$.
  (This is possible since $n = |G:H|>2$; for otherwise $H\subnormal G$
  contradicting $\core_G(H)=1$.) 
  Since $u \in U$ is arbitrary, we may assume
  $t_1 = t_k$ and $t_{x(1)}=t_j\neq t_{x(k)}$.  
  But this contradicts $t_1 d\, t_{x(1)}^{-1} = t_k d\, t_{x(k)}^{-1}$.
  Therefore, $x(1) = 1$.  The same argument shows that 
  $x(i) = i$ for each $1\leq i\leq n$, 
  and we see that
  $w=(d,\dots,d, x) \in N$ implies $x\in \ker \phi = 1$.  This puts $N$ below
  $D$, and the only normal subgroup of $U$ that lies 
  below $D$ is the trivial group.
\end{proof}
By the foregoing result we conclude that a class of groups that does
not include wreath products of the form $S\wr G$, where $S$ is an arbitrary
finite nonabelian simple group, is not a core-free interval enforceable class. 
The class of solvable groups is an example.


\subsection{Dedekind's rule}
\label{sec:dedekinds-rule}
When $A$ and $B$ are subgroups of a group $G$, by $AB$ we mean the set
$\{ a b \mid a\in A, b\in B\}$, and we write $A \join B$ or $\<A, B\>$ to denote
the subgroup of $G$ generated by $A$ and $B$.  
Clearly $AB \subseteq \<A,B\>$; 
equality holds if and only if $A$ and $B$ \emph{permute}, by which we mean $A B = B A$.

We will need the following well known result:\footnote{See~\cite[p.~122]{Rose:1978}, for example.}  
\begin{theorem}[Dedekind's rule]
  \label{lemma-dedekind}
Let $G$ be a group and let $A, B$ and $C$ be subgroups of $G$ with $A\leq B$.  Then,
\begin{align}
\label{eq:dedekind1}
A(C\cap B) &= AC \cap B,\qquad \text{ and }\\
\label{eq:dedekind2}
(C\cap B)A &= CA \cap B.
\end{align}
\end{theorem}

For $A \in \lb H, G \rb$, let $A^{\perp(H,G)}$ denote the set of
complements of $A$ in the interval $\lb H, G\rb$.  That is,
\[
A^{\perp(H,G)} := \{B \in \lb H, G \rb \mid A \cap B = H, \, \<A, B\> = G\}.
\]
Clearly $H^{\perp(H,G)} = \{G\}$ and $G^{\perp(H,G)} = \{H\}$.
Recall that an \emph{antichain} of a partially ordered set is a subset of
pairwise incomparable elements.

\begin{corollary}
\label{cor:dedekind1}
Let $A \in \lb H, G\rb$ and let 
$\cB$ be a nonempty subset of the set $A^{\perp(H,G)}$ of complements of $A$ in
$\lb H, G \rb$.  If every group in $\cB$ permutes with $A$, then $\cB$ is an
antichain. 
\end{corollary}
\begin{proof}
If $\cB$ is a singleton, the result holds
trivially. So assume $B_1$ and $B_2$ are distinct groups in $\cB$. 
We prove $B_1 \nleq B_2$.  Indeed, if $B_1 \leq B_2$, then 
Theorem~\ref{lemma-dedekind} implies
\[
B_1 = B_1H = B_1(A \cap B_2) = B_1A \cap B_2 = G \cap B_2 = B_2,
\]
which is a contradiction.
\end{proof}


\subsection{Parachute lattices}
\label{sec:parachute-lattices}
We now prove the equivalence of statements (B) and (C) 
of Section~\ref{sec:intro}.

\begin{theorem}
\label{thm-wjd-1}
The following statements are equivalent:
\begin{enumerate}
\item[(B)] Every finite lattice is isomorphic to
  an interval in the subgroup lattice of a finite group.

\item[(C)]
Suppose $n\geq 2$ and $\sL = \{L_1, \dots, L_n\}$ is a set of
finite lattices, at least two of which have more than two elements.
For each $i = 1, 2, \dots, n$, let $\sG_i$ denote the class that is
core-free interval enforcable by $L_i$. Then there exists a finite group $G \in
\bigcap\limits_{i=1}^n \sG_i$ such that for 
each $L_i \in \sL$ we have $L_i\cong \lb H_i, G \rb$ for some subgroup
$H_i$, where every $H_i \leq Y <G$ is core-free in $G$.
\end{enumerate}
\end{theorem}
\begin{remark}
By (C), the \acs{FLRP} would have a negative answer if we
could find a collection $\sG_1, \dots, \sG_n$ of \acs{cfIE} classes
such that $\bigcap\limits_{i=1}^n \sG_i$ is empty.
\end{remark}

\begin{proof}
Obviously (C) implies (B).  Assume (B) holds and assume the hypotheses of (C).
Construct a new lattice, denoted $\sP = \sP(L_1, \dots, L_n)$, as shown in
the Hasse diagram of Figure~\ref{fig:parachute} (a), where the bottoms of
the $L_i$ sublattices are atoms in $\sP$.
\begin{figure}[centering]
  \caption{The parachute construction.}
  \label{fig:parachute}
\begin{center}
{\scalefont{.8}
\begin{tikzpicture}[scale=0.7]

  \node (G) at (-8,0) [fill,circle,inner sep=1.2pt] {};
  \node (K) at (-11.5,-2) [fill,circle,inner sep=1.2pt] {};
  \node (K1) at (-9.9,-2.8) [fill,circle,inner sep=1.2pt] {};
  \node (K2) at (-8,-3.2) [fill,circle,inner sep=1.2pt] {};
  \node (Kn) at (-5.2,-2.2) [fill,circle,inner sep=1.2pt] {};
  \node (H) at (-8,-7) [fill,circle,inner sep=1.2pt] {};

\draw (-10,-1) node {$L_1$};
\draw (-9,-1.5) node {$L_2$};
\draw (-8,-1.6) node {$L_3$};
\draw (-6.5,-1) node {$L_n$};
\draw (-6.75,-2.8) node {$\dots$};

\draw (-8,-8.25) node {(a)};

\draw[semithick] 
   (K) to (H) to (K1)
   (K2) to (H) to (Kn);

\draw [semithick]  
   (G) to [out=-140,in=0] (K)
   (K)  to [out=55,in=185] (G)
   (G) to [out=-105,in=30] (K1)
   (K1) to [out=80,in=-140] (G)
   (G) to [out=-70,in=60] (K2)
   (K2)  to [out=110,in=-110] (G)
   (G) to [out=-10,in=110] (Kn)
   (Kn)  to [out=170,in=-50] (G);


  \node (Gr) at (1,0) [fill,circle,inner sep=1.2pt] {};
  \node (Kr) at (-2.5,-2) [fill,circle,inner sep=1.2pt] {};
  \node (K1r) at (-0.9,-2.8) [fill,circle,inner sep=1.2pt] {};
  \node (K2r) at (1,-3.2) [fill,circle,inner sep=1.2pt] {};
  \node (Knr) at (3.8,-2.2) [fill,circle,inner sep=1.2pt] {};
  \node (Hr) at (1,-7) [fill,circle,inner sep=1.2pt] {};

\draw (-1,-1) node {$L_1$};
\draw (0,-1.5) node {$L_2$};
\draw (1,-1.6) node {$L_3$};
\draw (2.5,-1) node {$L_n$};
\draw (2.25,-2.8) node {$\dots$};
\draw (1,-8.25) node {(b)};

\draw (Gr) node [above] {$G$}
    (Kr) node [left] {$K$}
    (K1r) node [left] {$K_1$}
    (K2r) node [left] {$K_2$}
    (Knr) node [right] {$K_n$}
    (Hr) node [right] {$H$};

\draw[semithick] 
   (Kr) to (Hr) to (K1r)
   (K2r) to (Hr) to (Knr);

\draw [semithick]  
   (Gr) to [out=-140,in=0] (Kr)
   (Kr)  to [out=55,in=185] (Gr)
   (Gr) to [out=-105,in=30] (K1r)
   (K1r) to [out=80,in=-140] (Gr)
   (Gr) to [out=-70,in=60] (K2r)
   (K2r)  to [out=110,in=-110] (Gr)
   (Gr) to [out=-10,in=110] (Knr)
   (Knr)  to [out=170,in=-50] (Gr);
\end{tikzpicture}
}
\end{center}
\end{figure}
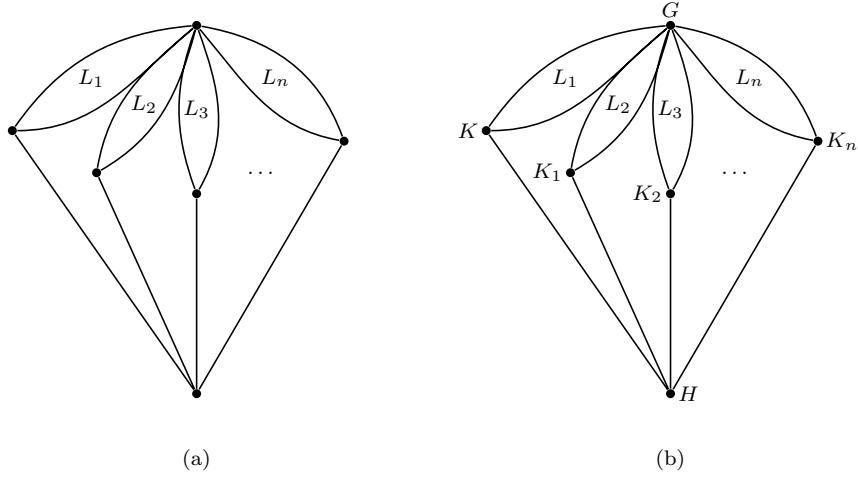
By (B), there exist groups $H <G$ with $\sP \cong \lb H,G \rb$.  We can assume $H$
is a core-free subgroup of $G$.  (If not, replace $G$ and $H$ with
$G/N$ and $H/N$, where $N=\core_G(H)$.)
Let $K, K_1, \dots, K_n$ be the subgroups in which $H$ is maximal
and for which $L_i \cong \lb K_i, G \rb,\; 1\leq i\leq n$. 
(Figure~\ref{fig:parachute} (b).)
We will prove that, for each $1\leq i\leq n$ every proper subgroup of 
$G$ that contains $K_i$ is core-free in $G$.
It then follows that $G\in \sG_i$ for all 
$1\leq i \leq n$, and so $G \in \bigcap\limits_{i=1}^n \sG_i$. 

Choose $Y$ such that $K_j \leq Y < G$.  We will prove $Y$ is core-free.
If $N = \core_G(Y)$ were nontrivial, then since $H$ is core-free,
we would have $K_j \leq NH \leq Y$.  Now, $NH$ permutes with all $X \in
\lb H, G\rb$, since for such $X$ we have $X NH = NX H = NHX$.  Therefore, 
if $N$ is nontrivial, then the set $(NH)^{\perp(H,G)}$, the
complements of $NH$ in $\lb H, G\rb$, forms an antichain by
Corollary~\ref{cor:dedekind1}. This contradicts the assumption that at least
two of the lattices $L_i$ have more than two elements.
\end{proof}

By a \emph{parachute lattice}, denoted $\sP(L_1, \dots, L_m)$,
we mean a lattice just like the one illustrated in 
Figure~\ref{fig:parachute}.
We identify some special group properties that are core-free interval
enforceable by a parachute lattice.
\begin{lemma}
\label{lemma-wjd-5}
 Let $\sP = \sP(L_1, \dots, L_n)$ with $n\geq 2$ and $|L_i|>2$ for at least two 
$i$, and suppose $\sP \cong \lb H, G \rb$ with $H$ core-free in $G$.  
\begin{enumerate}[(i)]
\item If $1\neq N \subnormal G$, then $NH = G$ and $C_G(N)=1$.
\item $G$ is subdirectly irreducible and nonsolvable.
\end{enumerate}
\end{lemma}
\begin{remark}
If $N$ is abelian, then $N \leq C_G(N)$, so (i) implies
that every nontrivial normal subgroup of $G$ is nonabelian.  
\end{remark}
\begin{proof}
(i)
Assume $1\neq N \subnormal G$. 
As above, we let $K_i$
denote the subgroups of $G$ corresponding to the atoms of $\sP$, and by the
same argument used to prove Theorem~\ref{thm-wjd-1},
we see that every subgroup $Y$ with $H \leq Y < G$ is core-free in $G$.  
Therefore, $NY=G$ for all $H \leq Y < G$. In particular, $NH=G$. 

To prove that $C_G(N)=1$, let $1\neq M \leq N$ be a minimal normal subgroup of
$G$ contained in $N$.  It suffices to prove $C_G(M)= 1$.
Note that $C_G(M) \subnormal N_G(M) =G$.  If $C_G(M)$ were nontrivial, then it
would follow by (1) that $C_G(M)H = G$.
Consider any $H< K < G$. Then $1 < M\cap K < M$ (strictly, by
Dedekind's rule). Now $M\cap K$ is normalized by $H$ and centralized
(hence normalized) by $C_G(M)$.  
Therefore, $M\cap K \subnormal C_G(M)H = G$, contradicting the minimality of
$M$.  

To prove (ii) we first show that $G$ has a unique minimal normal subgroup.  Let
$M$ be a minimal
normal subgroup of $G$ and let $N \subnormal G$ be any normal subgroup not 
containing $M$.  We show that $N = 1$.  Since both subgroups
are normal, the commutator subgroup
$[M,N]$
lies in the intersection $M\cap N$, which is trivial by the minimality of $M$.   
Thus, $M$ and $N$ centralize each other.  In particular,
$N \leq C_G(M) = 1$, by (i).
Finally, since $G$ has a unique minimal normal subgroup that is nonabelian, $G$
is nonsolvable.
\end{proof}

Given two group theoretical properties $\cP_1$ and $\cP_2$, we write
$\cP_1 \longrightarrow \cP_2$ to denote that a group $G$ has property $\cP_1$ only
if is also has property $\cP_2$. 
Thus, we clearly have 
\[
\quad \cP_1 \longrightarrow \cP_2 \quad
\Longleftrightarrow \quad \sG_{\cP_1}\subseteq
\sG_{\cP_2},
\]
where, as above, $\sG_{\cP_i}$ is the class of groups 
having property $\cP_i$. The conjunction $\cP_1 \meet \cdots \meet \cP_n$ corresponds to the class 
\[
\bigcap_{i=1}^n \sG_{\cP_i} = \{G \in \G \mid G \text{ has property $\cP_i$ for
  all $1\leq i\leq n$} \},
\]
and the following is an immediate corollary of the parachute construction:
\begin{corollary}
\label{cor:isle-prop-groups-1}
  If $\cP_1, \dots, \cP_n$ are cf-\IE\ properties,
  then so is $\cP_1 \meet \cdots \meet \cP_n$.
\end{corollary}

By Theorem~\ref{thm-wjd-1}, Lemma~\ref{lemma-wjd-5}, and
Corollary \ref{cor:isle-prop-groups-1}, we see that the \acs{FLRP} has a
positive answer (that is, statement (B) is true) if and only if for every finite
lattice $L$ there is a finite group $G$ satisfying all of the following:
\begin{enumerate}[(i)]
\item $L\cong \lb H, G \rb$;
\item $G$ is nonsolvable, nonalternating, and nonsymmetric;
\item $\core_G(Y) = 1$ for all $H\leq Y < G$;
\item $G$ has a unique minimal normal subgroup $M$, which satisfies $C_G(M) =
  1$;
in particular, $M$ is nonabelian and satisfies $MY = G$ for all $H\leq Y \leq G$.

\end{enumerate}

 \bibliography{wjd}
\bibliographystyle{plainurl}

\end{document}